\documentclass{amsart}
\usepackage{amsmath,amsthm}
\usepackage{amsfonts,amssymb}
\usepackage{enumerate}
\usepackage{accents,color}
\usepackage{graphicx}

\newcommand{\lvt}{\left|\kern-1.35pt\left|\kern-1.3pt\left|}
\newcommand{\rvt}{\right|\kern-1.3pt\right|\kern-1.35pt\right|}

\newtheorem{thm}{Theorem}
\newtheorem{cor}{Corollary}
\newtheorem{lem}{Lemma}

\newtheorem{THEO}{Theorem}

\newtheorem{QUE}{Question}

\theoremstyle{remark}

\makeatletter
\newcommand{\bddots}{%
  \mathinner{\mkern1mu\raise\p@\vbox{\kern7\p@\hbox{.}}\mkern2mu
    \raise4\p@\hbox{.}\mkern2mu\raise7\p@\hbox{.}\mkern1mu}}
\makeatother

 \def\NN{{\mathbb N}}

 \def\RR{{\mathbb R}}

\graphicspath{{./}}

\begin{document}

\title[Maclaurin coefficients of functions in the Laguerre-P\'olya class]
{Sign regularity of Maclaurin coefficients of functions in the Laguerre-P\'olya class}

\author{Dimitar K. Dimitrov}
\address{Departamento de Matem\'atica Aplicada\\
 IBILCE, Universidade Estadual Paulista\\
 15054-000 Sa\~{o} Jos\'e do Rio Preto, SP, Brazil.}
 \email{dimitrov@ibilce.unesp.br}
\author{Willian D. Oliveira}
\address{Departamento de Matem\'atica Aplicada\\
 IBILCE, Universidade Estadual Paulista\\
 15054-000 Sa\~{o} Jos\'e do Rio Preto, SP, Brazil.}
 \email{wdoliveira@ibilce.unesp.br}
 
 \thanks{Research supported by by the Brazilian foundations CNPq under
Grant 307183/2013-0 and FAPESP under Grant 2013/14881-9.}


\begin{abstract} We prove that the signs of the Maclaurin coefficients of a wide class of entire functions that belong to the 
Laguerre-P\'olya class posses a regular behaviour.     
\end{abstract}

\maketitle

\section{Introduction}
\setcounter{equation}{0}

A real entire function $\varphi$ is in the Laguerre-P\'olya class, written $\varphi \in \mathcal{LP}$, if 
\begin{equation}\label{LP} 
\varphi(z) = c z^m e^{a z^2+b z} \prod_{k=1}^\omega (1- z/x_k) e^{\lambda z/x_k},\ \ 0\leq \omega \leq  \infty,
\end{equation}
for some nonnegative integer $m$, $c, a, b \in \RR$, $a \leq 0$, $\lambda \in \{0,1\}$  and $x_k \in \RR \setminus\{0\}$ such that $\sum_k |x_k|^{-\lambda-1} < \infty$. 
When $\omega=0$ the product on the right-hand side of (\ref{LP}) is defined to be identically one. Otherwise 
the terms of the product are arranged according to the increasing order of $|x_k|$.

Observe that  $\varphi \in \mathcal{LP}$ if and only if $\varphi(z)=\exp(a z^2) \phi(z)$, where $a\leq 0$ and $\phi$ is a real entire function with 
real zeros of genus at most one.  

If $a=\lambda=0$ and $b\, \alpha_k\leq 0$  for every $k\in \mathbb{N}$ in the representation of $\varphi(z)$, the function  $\varphi(z)$ is said to belong to the Laguerre Polya Class of type $I$, denoted by $\varphi \in \mathcal{LP}I$. In other words, $\varphi \in \mathcal{LP}I$ if and only if either  $\varphi(z)$ or $\varphi(-z)$ can be represented in the form
$$\varphi(z) = c z^m e^{b z} \prod_{k=1}^\omega (1- z/x_k),$$
where $c\in \RR$, $m\in \NN \cup \{0\}$, $b\leq 0$, $x_k>0$ and $\sum_k x_k^{-1} < \infty$.

The functions in $\mathcal{LP}$ and only these are uniform limits on the compact sets of the complex 
plane (locally uniform limits) of hyperbolic polynomials, that is, real polynomials with only real zeros. Similarly, the functions in $\mathcal{LP}I$ are the locally uniform limits of hyperbolic polynomials with zeros of the same sign.

The class $\mathcal{LP}$ of functions was studied first by Laguerre \cite{Lag}  
and later then by P\'olya,  Jensen, Schur, Obrechkoff (see \cite{Lev1, Obre2003} and the references therein) and other celebrated mathematicians  in 
the beginning of the twentieth century because of the attempts to settle the Riemann hypothesis. The connection between the latter and the Laguerre-P\' olya Class is straightforward and we 
refer to \cite{CV, DR} for the details.

The properties of the functions in the Laguerre-P\'olya class are tightly connected with the notion of multiplier sequence. A real sequence $\{\gamma_n\}_{n=0}^\infty$ is called a multiplier sequence if for any hyperbolic polynomial 
\begin{equation}\label{P}
p(z)=a_0+a_1z+\dots+a_mz^m
\end{equation}
with zeros of the same sign, the polynomial
\begin{equation}\label{Q}
q(z)=a_0\gamma_0+a_1\gamma_1z+\dots+a_n\gamma_m z^m
\end{equation}
is a hyperbolic polynomial. Similarly, the sequence $\{\gamma_n\}_{n=0}^\infty$ is said to be a multiplier sequence of type $I$ if for any hyperbolic polynomial $p$ of the form (\ref{P}), the polynomial
$q$ given by (\ref{Q}) is a hyperbolic polynomial too. P\'olya and Schur \cite{PS} proved that the sequence $\{\gamma_n\}_{n=0}^\infty$ is a multiplier sequence if and only if the series
\begin{equation}
\label{ME}
\sum_{n=0}^\infty \frac{\gamma_n}{n!} z^n, 
\end{equation}
 represents an entire function in $ \mathcal{LP}$ class. In particular the sequence  $\{\gamma_n\}_{n=0}^\infty$ is a multiplier sequence of type $I$ if and only if the latter series represents an entire function in $ \mathcal{LP}I$ class. 
 
Therefore, given a real  entire function with Maclaurin expansion 
$$
\varphi(z) = \sum_{n=0}^\infty a_n z^n, 
$$
it is of interest to provide necessary and/or sufficient conditions in terms of the behaviour of the sequence $\{ a_n \}$ in order that 
$\varphi$ belongs to $\mathcal{LP}$. In the present note we establish some results of this nature.  We are interested if the signs of the coefficients $a_n$ exhibit some regular patterns.  For this purpose it is convenient to divide $\mathcal{LP}$ into three 
 subclasses. The first one is simply $\mathcal{LP}I$. The second class, denoted by $\mathcal{LP}^a$ consists of function in $\mathcal{LP}$ for which 
 the constant $a$ in the exponent in the representation (\ref{LP}) is nonzero. The third class consists of entire function in $\mathcal{LP}$ for which 
 $a=0$ but do not belong to $\mathcal{LP}I$ and is denoted by $\mathcal{LP}^0$.   
 
 It is straightforward that the Maclaurin coefficients of a nonpolynomial function $\varphi \in \mathcal{LP}I$ are either all of the same sign or their signs
 alternate. In fact, P\'olya and Schur \cite{PS} proved that a real entire function in the Laguerre-P\'olya class which exhibits this sign regularity of its coefficients is 
 necessarily in $\mathcal{LP}I$.     
 
However, Laguerre \cite{Lag} provided a beautiful example of a parametric family of entire functions in the Laguere-P\'olya class whose coefficients may exhibit signs form a 
rather irregular sequence for certain choices of the parameters. 
He proved that 
$$
\varphi(z)=\sum_{n=0}^{\infty} \frac{\cos{(\vartheta+n\theta)}}{n!}z^n
$$
is in $\mathcal{LP}$. In fact, it is straightforward to observe that  
$$
\varphi(z)= e^{z\cos \theta}\, \cos(\vartheta +z\sin\theta).
$$
We emphasise that $a=0$ and $\varphi \notin \mathcal{LP}I$ because its zeros are not of the same sign. Therefore $\varphi \in \mathcal{LP}^0$.
Varying the choices of $\vartheta$ and $\theta$ we observe a great variety of possibilities for the distribution of the signs in the sequence $\cos (\vartheta+n\theta)$, 
especially when $\theta/2\pi$ is an irrational number. Laguerre's example shows that one can hardly expect regularity in the distribution of the signs of the Maclaurin coefficients
of a function in $\mathcal{LP}^0$, similar to the one in $\mathcal{LP}I$. 
 
The situation changes dramatically when one considers the class  $\mathcal{LP}^a$. The presence of the factor $\exp(a z^2)$, $a<0$, seems to put order in 
the distribution of the signs of $a_n$.  Let us consider three illustrative examples. The signs of the first Maclaurin coefficients of $\exp(-z^2)\exp(z)$,  $\exp(-z^2)(1/\Gamma(z)$ 
and $\exp(-z^2)\cos\sqrt{z}$  are
$$\{1, 1, -1, -1, 1, 1, 1, -1, -1, 1, 1, -1, -1, 1, 1, -1, -1, 1, 1, -1, -1, -1, 1, 1, -1, -1, 1, 1\},$$
$$\{0, 1, 1, -1, -1, 1, 1, -1, -1, 1, 1, -1, 1, 1, -1, -1, 1, 1, -1, -1,1, 1, -1, -1, 1, 1, -1, -1\},$$
$$\{1, -1, -1, 1, 1, -1, -1, 1, 1, -1, -1, 1, 1, -1, -1, 1, 1, -1, -1,1, 1, -1, -1, 1, 1, -1, -1\},$$
respectively. Observe that the common pattern $\{1, 1, -1, -1\}$ appears frequently. A detailed analysis of these and a 
vast number of other function in $\mathcal{LP}^a$ led us to observe that this phenomenon occurs infinitely many times.
This pattern can be formalized  stating that the inequalities $a_{n-1}a_{n+1}<0$ hold asymptotically. These observations made us pose the following:
\begin{QUE} 
\label{Prob1}
Let $\varphi(z)=\sum_{n=0}^{\infty}a_n z^n$ be a real entire function with only real zeros. Is it true that $\varphi \in \mathcal{LP}^a$
if and only if  $0< \limsup\ n |a_n|^{2/n}<\infty$ and there exists $n_0 \in \mathbb{N}$, such that $a_{n-1} a_{n+1}\leq 0$ for all $n>n_0$?
\end{QUE}
We have substituted the strict inequalities $a_{n-1}a_{n+1}<0$ by $a_{n-1} a_{n+1}\leq 0$ in order to cover the possibility of occurence of zeros in the sequence $a_n$ which happens, for instance, when 
the entire function is either even or odd.  We establish a partial results towards an affirmative answer to Question \ref{Prob1}, proving that the inequalities $a_{n-1} a_{n+1}\leq 0$, together with adequate asymptotic rate of $a_n$ guarantees 
that a real entire function with real zeros indeed belongs to $\mathcal{LP}^a$. Formally, one of our main results reads as:
\begin{thm}
\label{Resultado 2} Suppose that the real function 
$$
\varphi(z)=\sum_{n=0}^{\infty}a_n z^n
$$ 
possesses only real zeros.  If there exist $n_0 \in \mathbb{N}$, such that $a_{n-1} a_{n+1}\leq 0$ for all $n>n_0$ and $0< \limsup_{n\to \infty}\ n |a_n|^{2/n}<\infty$, then $\varphi \in \mathcal{LP}^a$.
\end{thm}
We state a relevant result concerning the necessity statement of the question too.

\begin{thm} \label{Resultado 3.} If 
$$
\varphi(z)=\sum_{n=0}^{\infty}a_n z^n \in \mathcal{LP}^a
$$ 
can be represented in the form 
$$
\varphi(z)=\exp(az^2)P(z),
$$
where $a<0$ and $P$ is a hyperbolic polynomial, then $0< \limsup_{n\to \infty} \ n |a_n|^{2/n} <\infty$ and there exists $n_0 \in \mathbb{N}$, such that $a_{n-1} a_{n+1}\leq 0$ for every $n>n_0$.
\end{thm}
However, the pattern of the signs of the coefficients of functions in $\mathcal{LP}^a$ we observed before stating the above question, admits exceptions. Hence the answer to   
Question \ref{Prob1}  is negative. Moreover, exceptions occur rather frequently, for entire function in $ \mathcal{LP}^a$ of any possible order, as seen in the following:
\begin{thm} \label{Th3}
For every $\rho$ with $0\leq\rho\leq2$, there is an entire function $\Psi(z)$ of order $\rho_\Psi=\rho$, of the form 
\begin{equation}\label{fatorização3}
\psi(z)=\prod_{n=1}^{\infty}\left(1-\frac{z}{x_n}\right)\exp\left(\frac{\lambda z}{x_n}\right),
\end{equation}
such that, for every $a<0$, the function 
$$
\varphi(z)=\exp(az^2)\, \psi(z)=\sum_{n=0}^{\infty}a_n z^n
$$ 
is  in $\mathcal{LP}^a$ and possesses the property that for every $n_0 \in \mathbb{N}$, there is $n>n_0$, such that 
$a_{n-1} a_{n+1} > 0$.
\end{thm}

\section{Preliminaries}
Despite that most of the fact in this section are basic in the theory of entire functions and can be found in classical reading as \cite{Boas, Lev1}, we shall list them in the form we need in the proofs in order 
to make the reading relatively self-contained.

Given a  sequence of nonzero complex numbers $\{ z_n \}_{n\in \mathbb{N}}$, without accumulation points, we consider:
\begin{itemize}
\item[(i)] the associated canonic product 
$$
f(z)=\prod_{n=1}^{\infty}G\left(\frac{z}{z_n};p\right),
$$
where $p$ is the genus  of the canonic product, that is, the smallest integer for which 
$$
\sum_{n=1}^{\infty}\frac{1}{|z_n|^{p+1}}
$$ 
converges and
$$
G\left(u;p\right)=(1-u)e^{u+u^{2}/2+\dots+u^{p}/p};
$$
\item[(ii)] the exponent of convergence $\lambda$ of the sequence  as the  infimum of all positive numbers $t$ such that
$$
\sum_{n=1}^{\infty}\frac{1}{|z_n|^t}
$$
converges;
\item[(iii)] the superior density $\Delta$ of the sequence defined by
$$\Delta=\limsup_{r\to \infty} \frac{n(r)}{r^{\lambda}},
$$
where $n(r)=\# \{ n : |z_n|\leq r \}$.
\end{itemize}

As it is well known, an entire function 
$$
f(z)=\sum_{n=0}^{\infty}a_n z^n
$$ 
is of order $\rho_f$ if
$$
\limsup_{r\to \infty}\frac{\log \log M(r)}{\log r}=\rho_f,
$$
where $M(r)=M_f(r)=\max_{|z|=r}|f(z)|$. If $\rho_f<\infty$, then $f$ is said to be of  type $\sigma_f$ if
$$
\limsup_{r\to \infty}\frac{\ \log M(r)}{r^{\rho}}=\sigma_f.
$$

When $0<\rho_f<\infty$ the order and the type are determined in terms of the coefficients via 
\begin{eqnarray}
\rho_f & = & -\limsup_{n\to \infty} \frac{n\log n }{\log |a_n|}, \nonumber\\
\ & & \label{rscoef} \\
\sigma_f & = & \limsup_{n\to \infty} \frac{n\sqrt[n]{|a_n|^{\rho_f}}}{e\rho_f}. \nonumber
\end{eqnarray}
Borel's theorem claims that the order of a canonic product is equal to the exponent of convergence of its zeros.

 The category $\kappa_f$ of the entire function $f$ is the pair of its order and type, 
that is $\kappa_f=(\rho_f,\sigma_f)$.  The entire functions are partially ordered (denoted $\preceq$) according to alphabetic order of their categories. Precisely, $f$ and $g$ belong to the same category provided 
both their orders and types coincide. Otherwise, if $\rho_f>\rho_g$ then $f$ belongs to a higher category. Finally, if $\rho_f=\rho_g$ but $\sigma_f>\sigma_g$ then $f$ is of a higher category. 
The so-called Theorem of categories (see \cite[Theorem 12, p. 23]{Lev1}) states if the entire functions $f$ and $g$ belong to distinct categories, then the product $fg$ belongs to the category of the factor with 
the higher category. In other words,
\begin{equation}
\label{cat}
\kappa_{fg} = \max\{ \kappa_f, \kappa_g\}.
\end{equation}

Hadamard's theorem claims that every entire function $f$ of finite order $\rho_f$ can be represented in the form 
\begin{equation}
\label{Had}
f(z)=z^k \exp(P(z))\prod_{n=1}^{\infty}G\left(\frac{z}{z_n};p\right),
\end{equation}
where $k\in \mathbb{N}\cup\{0\}$, $P$ is an algebraic polynomial of degree not exceeding $\rho_f$,  and $z_n$ are the zeros of $f$ distinct from the origin.

We state the classical theorem of Lindel\"of in its complete form, as in \cite[Theorem 15, p. 28]{Lev1}. 
  
\begin{THEO} {\rm (Lindel\"of)} 
\label{lindelöf} Let $f$ be an entire function with Hadamard product {\rm(\ref{Had})}, of finite order $\rho_f$, type $\sigma_f$, genus $p$ of its canonic product and superior density of its zeros $\Delta_f$. Then the following hold:
\begin{itemize}
\item[(i)] If $\rho_f$ is not integer, then $\Delta_f$ and $\sigma_f$ are simultaneously equal to either $0$, a finite number, or $\infty$.
\item[(ii)] If $\rho_f$ is an integer and $\rho_f> p$, then $f$ is of type $\sigma_f=|\alpha_{\rho_f}|$, where $\alpha_{\rho_f}$ is the coefficient of  $z^{\rho_f}$ in the expansion of the polynomial $P(z)$.
\item[(iii)] If $\rho_f=p$ and 
$$
\delta_f(r) = \left|\alpha_{\rho_f}+\frac{1}{\rho_f} \sum_{|z_n|<r} z_n^{-\rho_f}\right|, \ \ \ \bar{\delta}_f=\limsup_{r\to \infty}\delta_f(r), \ \ \ \gamma_f=\max(\Delta_f,\bar{\delta}_f),
$$
then $\sigma_f$ and $\gamma_f$ are simultaneously equal to either $0$, a finite number, or $\infty$.
\end{itemize}
\end{THEO}


Finally, the classical Hermite-Biehler theorem states that all zeros of the algebraic polynomial
$$
r(z)=u(z)+iv(z),
$$
where $u(z)$ and $v(z)$ are polynomials with real coefficients, belong to one of the open semi-planes determined by the real axis, if and only if both
$u$ and $v$ are hyperbolic polynomials and their zeros interlace.

\section{Proofs} 
We begin this section with a technical lemma which shows how to determine the order and the type of an entire function from the asymptotic behaviour of 
its Taylor coefficients. The result might be of independent interest and the idea of the proof goes back to \cite[Theorem 2, p. 4]{Lev1}.
\begin{lem}\label{lema resultado2} Let $\rho>0$ and $f(z)=\sum_{n=0}^{\infty}a_n z^n$ be an entire function, such that $0< \limsup_{n\to \infty} (n\sqrt[n]{|a_n|^\rho})<\infty$. Then 
$\rho_f=\rho$ and $0<\sigma_f<\infty$.
\end{lem}
\begin{proof} The proof goes by {\em reductio ad absurdum}. First we assume  that $\rho_f>\rho$. It follows from $\limsup_{n\to \infty} (n\sqrt[n]{|a_n|^\rho})<\infty$ that there is $K>0$, such that
$$
\frac{n\sqrt[n]{|a_n|^\rho}}{\rho e}<K \ \ \mathrm{for} \ \ n>n(K),
$$
or equivalently,
\begin{equation}\label{2}
|a_n|<\left(\frac{\rho eK}{n}\right)^{n/\rho} \ \ \mathrm{for} \ \ n>n(K).
\end{equation}
Let $|z|=r$, where $r>r(K)$ is large enough so that the inequality $N(r)=[2^\rho \rho eK r^{\rho}]>n(K)$ holds. Here $[\cdot]$ stands for the integer part. Then we employ the estimate (\ref{2}) to 
obtain 
\begin{eqnarray*}
|f(z)| & \leq & \sum_{n=0}^{N(r)} |a_n| r^{n} + \sum_{n=N(r)+1}^{\infty} |a_n|r^{n} \\
\ & \leq & \sum_{n=0}^{N(r)}|a_n|r^{n}+2^{-N(r)}\\
\ & < & (N(r)+1)\max_{n}\{a_nr^{n}\}+2^{-N(r)}.
\end{eqnarray*}
Hence, 
$$
M_f(r)<( 2^\rho \rho eK r^{\rho}+1)\, \max_{n}\{a_nr^{n}\}+1 \ \ \mathrm{for} \ \ r>r(K).
$$
Observe that $f$ is not a polynomial because we assumed that $\rho_f>\rho>0$. Then the Cauchy estimate  yields that $M_f(r)$ increases faster than any power of $r$. The latter inequality implies that $\max_{n}\{a_nr^{n}\}$ increases faster than any power of $r$ too. This means that the index $n$, where the maximum is attained increases going to infinity as $r\to \infty$. Redefining $r(K)$, if necessary, so that the index $n$ where  $\max_{n}\{a_nr^{n}\}$ occurs be such that $n>n(K)$,  we obtain
$$
\max_{n}\{a_nr^{n}\}\leq \left(\frac{\rho eK}{n}\right)^{n/\rho}r^{n}, \ \ \ \ \ \ r>r(K).
$$
A straightforward analysis show that the maximum on the right-hand side of the latter is attained at $n=\rho K r^\rho$. Thus 
$$\max_{n}\{a_nr^{n}\}\leq e^{K r^{\rho}},\ \ \ \ \ \ r>r(K).$$
Then
$$
M_f(r)<(2^\rho \rho eK r^{\rho}+2)e^{Kr^{\rho}}\ \ \mathrm{for} \ \ r>r(K),
$$
and, increasing $r(K)$ if necessary, 
\begin{equation}\label{contradição1}
M_f(r)<e^{2K r^{\rho}},\ \ \ \ \ \ r>r(K).
\end{equation}
On the other hand, the definition of the order of an entire function and the assumption that $\rho_f>\rho$ show that we may chose $\epsilon_K>0$ in such a way that $\rho_f-\epsilon_K>\rho$ and $r^{\rho_f-\epsilon_K-\rho}>2K$ for every $r>r(\epsilon_K)$ and 
 $$
 M_f(r)>\exp (r^{\rho_f-\epsilon_k})=\exp (r^{\rho_f-\epsilon_K-\rho}r^{\rho})>\exp (2Kr^{2}), \ \ \ \ \ \ r>r(\epsilon_K).
 $$
The latter contradicts (\ref{contradição1}). 

Assume now that $\rho_f<\rho$. Then, given $\epsilon>0$, there exists $r(\epsilon)$ such that  
$$
M_f(r)\leq \exp(\epsilon r^\rho),  \ \ \ \ \ \ r>r(\epsilon).
$$
By the Cauchy estimate
$$
|a_n|  \leq \frac {M_f(r)} {r^n} \leq \frac { \exp(\epsilon r^\rho)} {r^n},  \ \ \ \ \ \ r>r(\epsilon).
$$
The minimum of $r^{-n}\exp(\epsilon r^\rho)$, when $r>0$, is attained at $r=(n/\epsilon \rho)^{1/\rho}$. Choose now $n(\epsilon)$ such that $(n/\epsilon \rho)^{1/\rho}>r(\epsilon)$ for all $n>n(\epsilon)$, to obtain
$$
|a_n|  \leq \left(\frac{\rho \epsilon e}{n}\right)^{n/\rho}, \ \ \ \ \ \ n>n(\epsilon),
$$
which is equivalent to  
$$n\sqrt[n]{|a_n|^\rho}<\rho\epsilon e, \ \ \ \ \ \ n>n(\epsilon).$$
Since the above reasonings hold for any $\epsilon>0$, then
$$
\limsup_{n\to \infty} (n\sqrt[n]{|a_n|^\rho})=0.
$$
This a contradiction to the hypothesis that $\limsup_{n\to \infty} (n\sqrt[n]{|a_n|^\rho})>0$. 

Therefore, $\rho_f=\rho$. The fact that  $0<\sigma_f<\infty$ follows from (\ref{rscoef}).
\end{proof}

\begin{proof} {\em (of Theorem \ref{Resultado 2})} 
By Lemma \ref{lema resultado2}, $\rho_\varphi=2$ and $0<\sigma_\varphi<\infty$.  Hadamard's factorisation theorem implies that $\varphi$ can be represented as
\begin{equation}\label{fatorização}
\varphi(z)=cz^k \exp(az^2+bz)\prod_{n=1}^{\infty}G\left(\frac{z}{x_n};p\right),
\end{equation}
where  $k\in \mathbb{N}\cup \{0\}$,  $a,b,c\in \mathbb{R}$,  and $x_n$ are the zeros of $\varphi$ distinct from $0$. We shall prove that the series 
\begin{equation}\label{série}
\sum_{n=1}^{\infty}\frac{1}{x_n^{2}}
\end{equation}
converges. Let $\rho_\Pi$ and $\sigma_\Pi$ be the order and the type of the canonic product in (\ref{fatorização}). Since $\rho_\varphi=2$, then (\ref{cat}) implies 
\begin{equation}\label{contr}
\kappa_{\Pi} \preceq (2,\sigma_\varphi).
\end{equation}
Assume that the series (\ref{série}) diverges. Then the exponent of  convergence $\lambda_\Pi$ of the canonic product in (\ref{fatorização}) obeys  $\lambda_\Pi \geq 2$. It follows from  
the above mentioned theorem of Borel and (\ref{contr}) that  $\lambda_\Pi=\rho_\Pi=2$. Since (\ref{série}) diverges, then we must have $p=2$. Hence $\rho_\Pi=p$ and, by item (iii) of Teorema \ref{lindelöf},
$$
\sigma_\Pi = \limsup_{r\to \infty}\left(\frac{1}{2}\sum_{|x_n|<r}\frac{1}{x_n^{2}}\right)=\infty.
$$
Therefore,  $\kappa_\Pi=(2,\infty)$, so that, by (\ref{contr}),  $\kappa_\varphi=(2,\infty)$. This contradicts $\sigma_\varphi<\infty$. Hence, the series (\ref{série}) converges.

The convergence of  (\ref{série}) implies that $p\leq1$ and that (\ref{fatorização}) can be written in the form
\begin{equation}\label{fatorização2}
\varphi(z)=cz^k\exp(az^2+bz)\prod_{n=1}^{\infty}\left(1-\frac{z}{x_n}\right)\exp\left(\frac{\lambda z}{ x_n}\right),
\end{equation}
with $k\in \mathbb{N}\cup\{0\}$, $a, b, c\in \mathbb{R}$, where $x_n \in \mathbb{R}$ for every $n\in \mathbb{N}$,  $\lambda \in \{0,1\}$ and $\sum_{}^{}(1/x_{n})^{\lambda+1}<\infty$. 
It remain to prove only that  $a<0$ in (\ref{fatorização2}). Assume the contrary, that  $a\geq0$. 

However, there is $n_0$, such that $a_{n-1}a_{n+1}\leq 0$ for all $n>n_0$. Then we can write $\varphi(z)=\sum_{n=0}^{\infty}a_n z^n$ in the form
\begin{equation}\label{decomposição1}
\varphi(z)=P_{n_0}(z)+g(z),
\end{equation}
with a polynomial $P_{n_0}(z)$ of degree $n_0$ and $g$ an entire function of the form 
$$
g(z)=\sum_{n=0}^{\infty}\,sgn(b_n)\,|b_n|\, z^n\ \ \ \mathrm{with}\ \ \ sgn(b_{n-1})\, sgn( b_{n+1})\leq 0\ \ \ \mathrm{for\ all}\ \  n>0.
$$ 
If $t>0$ then 
\begin{eqnarray*} 
\left|g(it)\right| & = & \left| \sum_{n=0}^{\infty}sgn(b_n)|b_n|(it)^n \right|=\left| sgn(b_0)\sum_{n=0}^{\infty} |b_{2n}| t^{2n} + i\, sgn(b_1)\sum_{n=0}^{\infty} |b_{2n+1}| t^{2n+1} \right |\\
\ & = & \sqrt{\left(\sum_{n=0}^{\infty} |b_{2n}| t^{2n}\right) ^2+\left(\sum_{n=0}^{\infty} |b_{2n+1}| t^{2n+1}\right)^2}\\
\ & \geq & \frac{1}{\sqrt{2}} \sum_{n=0}^{\infty} |b_{n}|t^{n}\geq \frac{1}{\sqrt{2}} \max_{|z|= t} |g(z)|=\frac{1}{\sqrt{2}}M_g(t).
\end{eqnarray*}
Since $\rho_\varphi=2>p$, then item (ii) of Theorem \ref{lindelöf} yields $0<\sigma_\varphi=a$. Now (\ref{rscoef}) allows us to conclude that $\rho_g=2$ and $\sigma_g=a>0$. 
This implies that there is $t_0>0$, such that
$$
M_g(t)>\exp ((a/2)t^2) \ \ \mathrm{and} \ \  \frac { \sqrt{2}|P_{n_0}(it)|}{\exp ((a/2)t^2)}<\frac{1}{2}\ \ \mathrm{for\ every} \ \  t>t_0.
$$
Therefore
\begin{eqnarray*}
|\varphi(it)| & \geq & |g(it)|-|P_{n_0}(it)|\geq \frac{1}{\sqrt{2}}M_g(t)-|P_{n_0}(it)|\\
\ & \geq  & \frac{1}{\sqrt{2}}\exp ((a/2)t^2) \left( 1-\frac{ \sqrt{2}|P_{n_0}(it)|}{\exp ((a/2)t^2)}\right).
\end{eqnarray*} 
Finally we conclude that 
\begin{equation}\label{contradição3}
|\varphi(it)|> \frac{1}{2\sqrt{2}}\exp ((a/2)t^2) \ \ \mathrm{for\ every} \ \  t>t_0.
\end{equation}
The  convergence of the series (\ref{série}) shows that there are two possibilities: either $\rho_\Pi<2$ or $\rho_\Pi=2$ and, in the latter case, $\sigma_\Pi=0$ because of item (ii) of Theorem \ref{lindelöf}. 
In both cases, there is $t_1>0$, such that  
$$
\left| \prod_{n=1}^{\infty}\left(1-\frac{it}{a_n}\right)\exp\left(\frac{\lambda i t}{a_n}\right)\right|<\exp((a/2)t^2)\ \ \mathrm{and} \ \  |c|t^k\exp(-(a/2)t^2)<1 \ \ \mathrm{for} \ \  t> t_1.
$$
Using this inequality in (\ref{fatorização2}) we obtain  
$$
|\varphi(it)|=|c|t^k\exp(-at^2)\left| \prod_{n=1}^{\infty}\left(1-\frac{it}{\alpha_n}\right)\exp\left(\frac{\lambda i t}{\alpha_n}\right)\right|<|c|t^k\exp(-(a/2)t^2)<1, \ \ t> t_1.
$$
This contradicts  (\ref{contradição3}). Therefore $a<0$ in (\ref{fatorização2}) and $\varphi\in \mathcal{LP}^a$.
\end{proof}

\begin{cor}\label{resultado1} Let $\varphi(z)=\sum_{n=0}^{\infty}a_n z^n$ be an even (odd) function with only real zeros. Then $\varphi \in \mathcal{LP}^a$ if and only if $0< \limsup (n\sqrt[n]{|a_n|^2})<\infty$ and $a_{n-1}a_{n+1}\leq 0$ for every $n\in\mathbb{N}$.
\end{cor}
\begin{proof} If  $0< \limsup (n\sqrt[n]{|a_n|^2})<\infty$ and $a_{n-1}a_{n+1}\leq 0$ for every $n\in\mathbb{N}$, then, by Theorem \ref{Resultado 2}, $\varphi \in \mathcal{LP}^a$. 

Suppose that $\varphi \in \mathcal{LP}^a$ is even (odd). It can be written in the form 
$$
\varphi(z)=cz^k\exp(az^2)\prod_{n=1}^{\infty}\left(1-\frac{z^2}{x_n^2}\right).
$$
It suffices to apply Cauchy's multiplication formula to see that $a_{n-1}a_{n+1}\leq 0$ for all $n\in\mathbb{N}$. Since $\sum_{n=1}^{\infty}(1/x_n^2)<\infty$, then  item (ii) of Theorem \ref{lindelöf} yields $\sigma_\varphi=-a$. 
Finally,   (\ref{rscoef}) implies $0< \limsup (n\sqrt[n]{|a_n|^2})<\infty$.
\end{proof}

\begin{proof} {\em (of Theorem \ref{Resultado 3.})}\ Since $\rho_\varphi=2$ and $\sigma_\varphi=-a$, then (\ref{rscoef})
yields 
$$
0< \limsup (n\sqrt[n]{|a_n|^2}) < \infty.
$$ 
Let  
$$
P(z)=\sum_{n=0}^{2m}\ b_n\, z^n,
$$
with the obvious convention that if $P$ is of odd degree, then $b_{2m}=0$.
Consider another polynomial, 
$$
P_1(z)=a^mb_0+a^{m-1}b_2 z+\dots+b_{2m}z(z-1)\dots(z-m+1).
$$
There exists $k_0\in \mathbb{N}$, such that $P_1(x)$ is strictly positive or strictly negative for $x>k_0$. Let $n=2k+1$ be an odd number with $k>k_0$ and $k\geq m$. The Cauchy multiplication formula  yields
\begin{eqnarray*}
a_{n-1} & = & \sum_{j=0}^{m}\frac{a^{k-j}}{(k-j)!}b_{2j}=\frac{a^{k-m}}{k!}\sum_{j=0}^{m}a^{m-j}\, b_{2j}\, k(k-1)\dots(k-j+1)\\
\ & =& \frac{a^{k-m}}{k!}P_1(k),\\
a_{n+1} & = & \sum_{j=0}^{m}\frac{a^{k+1-j}}{(k+1-j)!} b_{2j}=\frac{a^{k+1-m}}{(k+1)!}\sum_{j=0}^{m}a^{m-j}\, b_{2j}\, (k+1)(k)\dots(k-j+2)\\
\ & = & \frac{a^{k+1-m}}{(k+1)!}P_1(k+1).
\end{eqnarray*}
Having in mind that $a<0$ and $P_1(k)$ and $P_1(k+1)$ have equal signs when $k>k_0$, we see that
$$
a_{n-1} a_{n+1}=\frac{a^{2(k-m)+1}}{k!(k+1)!}P_1(k)\, P_1(k+1)\leq 0, 
$$
for all odd indices $n>2k_0+1$. When $n$ is even, it suffices to apply an analogous argument but employing 
the polynomial  
$$
P_2(x)=a^{m-1}b_1+a^{m-2}b_3(x-1)+\dots+b_{2m-1}(x-1)(x-2)\dots(x-m+1)
$$
instead of $P_1$. Again, there is $n_0$, such that $a_{n-1} a_{n+1}\leq 0$ for all $n>n_0$. 
\end{proof}

Now we establish a lemma whose proof might be of independent interest because it is simpler than the proof of the 
corresponding general claim about functions in the Hermite--Biehler space. 

\begin{lem}\label{Hermite-Biehler generalizado} Let $\varphi(z)=\sum_{n=0}^{\infty}a_n z^n$ be an entire function with only real zeros, all of the same sign and $0\leq \rho_\varphi<1$. Then all the zeros of both 
\begin{eqnarray*}
\varphi^o(z) & = & \sum_{n=0}^{\infty}a_{2n+1} z^{2n+1},\\
\varphi^e(z) & = & \sum_{n=0}^{\infty}a_{2n} z^{2n}
\end{eqnarray*}
 are purely imaginary.
\end{lem}
\begin{proof}  Since $0\leq \rho_\varphi<1$, then the Hadamard product of $\varphi$ is
$$
\varphi(z)=\varphi(0) \prod_{n=1}^{\infty}\left(1-\frac{z}{x_n}\right).
$$
Clearly, $\varphi$ is a local uniform limit of hyperbolic polynomials
$$
P_m(z)=\varphi(0) \prod_{n=1}^{m}\left(1-\frac{z}{x_n}\right) = \sum_{n=0}^{\infty} a_{2n,m} z^{2n}+\sum_{n=0}^{\infty} a_{2n+1,m} z^{2n+1},
$$
with $ a_{n,m}\in \mathbb{R}$ and $a_{2n,m}=a_{2n+1,m}=0$, when $2n>m$. Then obviously
$$
P_m(iz)=\varphi(0) \prod_{n=1}^{m}\left(1-\frac{iz}{x_n}\right)=\sum_{n=0}^{\infty}(-1)^n a_{2n,m} z^{2n}+i\sum_{n=0}^{\infty} (-1)^na_{2n+1,m} z^{2n+1}.
$$
For each fixed $m\in\mathbb{N}$, the zeros of the polynomial $P_m(iz)$  are $-i x_1, \dots,-i x_m$ and they posses imaginary parts of the same sign because the zeros $x_n$ of $\varphi$ are all of the same sign. By the Theorem of Hermite-Biehler, 
the zeros of both $\sum_{n=0}^{\infty}(-1)^n a_{2n,m} z^{2n}$ and $\sum_{n=0}^{\infty} (-1)^n a_{2n+1,m} z^{2n+1}$ are all real \cite{Obre2003}. Therefore, the zeros of $\sum_{n=0}^{\infty} a_{2n,m} z^{2n}$ and  $\sum_{n=0}^{\infty} a_{2n+1,m} z^{2n+1}$ are 
all purely imaginary. Since these polynomials converge locally uniformly to the functions $\varphi^o(z)$ and $\varphi^e(z)$, when $m\to \infty$, it follows from the Theorem of Hurwitz, that the zeros of $\varphi^o(z)$ e $\varphi^e(z)$ are purely imaginary too. 
\end{proof}

\begin{proof}{\em (of Theorem \ref{Th3})}.   First we shall prove that if  $\psi$ is an entire function with the property that at least one of the functions 
$$
\psi^o(z)=\sum_{n=0}^{\infty} b_{2n+1} z^{2n+1}\ \ \mathrm{and}\ \ \psi^e(z)=\sum_{n=0}^{\infty} b_{2n} z^{2n}
$$ 
has an infinite number of purely imaginary zeros, then, for any $a<0$, the function 
$$
\varphi(z)=\exp(az^2)\, \psi(z)=\sum_{n=0}^{\infty}a_n z^n 
$$ 
obeys the property that for every $n_0 \in \mathbb{N}$, there exists $n>n_0$, such that $a_{n-1} a_{n+1} > 0$. 
Indeed, for a given $a<0$, at least one of the functions  
$$
\varphi^o(z)=\sum_{n=0}^{\infty}a_{2n+1}z^{2n+1}=\exp(az^2)\sum_{n=0}^{\infty} b_{2n+1} z^{2n+1},
$$
$$
\varphi^e(z)=\sum_{n=0}^{\infty}a_{2n}z^{2n}=\exp(az^2)\sum_{n=0}^{\infty} b_{2n} z^{2n}
$$
is not a polynomial and possesses infinitely many purely imaginary zeros. Suppose, without loss of generality, that this is $\varphi_2$, that is, 
$\varphi^e(z)=\sum_{n=0}^{\infty}a_{2n}z^{2n}$ is an entire function and it is not a polynomial. Assume that there exists $n_0 \in \mathbb{N}$ such that $a_{n-1} a_{n+1}\leq 0$ for all $n>n_0$. 
Then we may represent it as
\begin{equation}\label{decomposição2}
\varphi^e(z)=P_{n_0}(z)+h(z),
\end{equation}
where $P_{n_0}(z)$ is a  polynomial of degree $n_0$ and 
$$
h(z)=\sum_{n=0}^{\infty}sgn(c_{2n}) |c_{2n}| z^{2n}
$$ 
is an entire function  with $sgn(c_{2n})sgn( c_{2(n+1)})\leq 0$ for every $n\geq0$. 
Then for any $t\in \mathbb{R}$ we have
\begin{eqnarray}
 \left|h(it)\right| & = & \left| \sum_{n=0}^{\infty}sgn(c_{2n})|c_{2n}|(it)^n \right|=\left| sgn(c_0)\sum_{n=0}^{\infty} |c_{2n}| |t|^{2n} \right   |= \sum_{n=0}^{\infty} |c_{2n}| |t|^{2n} \nonumber \\
 \ & \geq & \max_{|z|= |t|} |h(z)|=M_h(|t|). \label{desigualdade}
 \end{eqnarray}
Let $\{it_m\}$ be the sequence of the purely imaginary zeros of $\varphi^e$. Then  (\ref{decomposição2}) and (\ref{desigualdade}) imply 
$$M_h(|t_m|)\leq |P_{n_0}(it_m)|,$$
for all $m\in \mathbb{N}$. The fact that $\varphi^e$ is an entire function shows that $|t_m|\to \infty$ as $m\to \infty$. Thus, there exists $m_0>0$, such that $|P_{n_0}(it_m)|<|t_m|^{n_0+1}$, for all $m>m_0$. 
On the other hand, it follows from the Cauchy estimate that  
$$|c_{2n}|\leq \frac{M_h(|t_m|)}{t_m^{2n}}\leq \frac{|P_{n_0}(it_m)|}{t_m^{2n}}\leq \frac{|t_m|^{n_0+1}}{t_m^{2n}}, \ \ \ \ \ \ m>m_0.$$
However, the latter implies that $c_{2n}=0$ for all $n>(n_0+1)/2$. Then, by  (\ref{decomposição2}), the function $\varphi_2$ is a polynomial, a contradiction. 

Back to the proof of the theorem itself,  suppose first that $\rho=0$. Consider a sequence $x_n$ of real positive numbers with exponent of convergence $\lambda=0$. We may consider, for instance, the sequence $x_n=e^n$.
Then the canonic product  
$$\Pi_1(z)= \prod_{n=1}^{\infty}\left(1-\frac{z}{x_n}\right)$$
is of order $\rho_{\Pi_1}=0$. By Lemma \ref{Hermite-Biehler generalizado} the functions $\Pi_1^o(z)$ e $\Pi_1^e(z)$ have only purely imaginary zeros. Set $\psi(z)=\Pi_1(z)$. By the claim we have proved in the beginning the function $\psi$ is such that 
$\rho_\psi=\rho=0$ and,  for each $a<0$, the function $\varphi(z)=\exp(az^2)\psi(z)=\sum_{n=0}^{\infty}a_n z^n \in \mathcal{LP}^a$ possesses  the property that for every $n_0 \in \mathbb{N}$, there exists $n>n_0$, such that $a_{n-1} a_{n+1} > 0$. 

In the case $0<\rho\leq2$, consider the sequence $y_n$ of real numbers,  symmetric with respect to the origin, with exponent of  convergence $\lambda=\rho$ and such that $\sum (1/|y_n|^2)$  converges. We may consider, for instance, 
the sequence, $y_{2n-1}=-(n\log (n+1) )^{1/\rho}$ and $y_{2n}=(n\log (n+1) )^{1/\rho}$. Then the canonic product 
$$
\Pi_2(z)= \prod_{n=1}^{\infty}\left(1-\frac{z}{y_n}\right)\exp \left( \frac{\lambda z}{y_n} \right), 
$$
with $\lambda\in\{0,1\}$, represents an even entire function. Choose now $\psi(z)=\Pi_2(z)\Pi_1(z)$. Since $\psi^o(z)=\Pi_2(z)\Pi_1^o(z)$ and $\psi^e(z)=\Pi_2(z)\Pi_1^e(z)$, the real and the imaginary parts of $\psi$ possess an infinite number of purely imaginary zeros.  Therefore $\psi(z)=\Pi_2(z)\Pi_1(z)$ is such that, for each $a< 0$, the function $\varphi(z)=\exp(az^2)\psi(z)=\sum_{n=0}^{\infty}a_n z^n$ obeys the property that for every $n_0 \in \mathbb{N}$, there exists $n>n_0$, such that $a_{n-1} a_{n+1} > 0$. 
 Finally, 
observe that arranging the sequence $\{ z_n \}$ of the zeros of $\Pi_1(z)\Pi_2(z)$ in an increasing order of their absolute values, 
$$
\psi(z)=\prod_{n=1}^{\infty}\left(1-\frac{z}{z_n}\right)\exp \left( \frac{\lambda z}{z_n} \right),
$$
with $\lambda\in\{0,1\}$ and $\sum (1/|z_n|^{\lambda+1})$ convergent. Therefore, $\varphi(z)=\exp(az^2)\psi(z)\in \mathcal{LP}^a$.
\end{proof}

\end{document}